\long\def\skipit#1{} 
\def\={\,=\,}
\def\+{\,+\,}
\def\-{\,-\,}

\newcommand{\eu}{{\mathfrak{eu}}}  
\newcommand{\Eu}{{\mathscr{E}}}  

\documentclass[12pt]{amsart}
\usepackage{pstricks-add}
\usepackage{amsfonts}
\usepackage{amsmath}
\usepackage{amssymb}
\usepackage{ifthen,calc}
\usepackage{color}
\usepackage{epsfig}
\usepackage{graphicx}
\usepackage{epstopdf}
\usepackage{epsfig}
\usepackage{latexsym}
\usepackage{multicol}
\usepackage[sc]{mathpazo}
\usepackage[format=hang, margin=10pt]{caption}
\usepackage[mathscr]{eucal}

\newcounter{hours}
\newcounter{minutes}
\newcommand{\printtime}{
	\setcounter{hours}{\time/60}%
	\setcounter{minutes}{\time-\value{hours}*60}
	\ifthenelse{\value{hours}<10}{0}{}\thehours:%
	\ifthenelse{\value{minutes}<10}{0}{}\theminutes}

\numberwithin{equation}{section}
\numberwithin{figure}{section}
\numberwithin{table}{section}

\newtheorem{thm}{Theorem}[section]
\newtheorem{cor}[thm]{Corollary}
\newtheorem{lemma}[thm]{Lemma}

\newtheorem{J-com}{JG-comment}[section]
\theoremstyle{definition}
\newtheorem{example}{Example}[section]

\newtheorem{problem}[thm]{Problem}
\newtheorem{conj}[thm]{Conjecture}
\newtheorem{rem}[thm]{Remark}

\def\dsum{\displaystyle\sum}

\keywords{ partial dual,  partial Petrial, even-interpolating, interpolating}
\begin{document}

\title{ Parallel edges in ribbon graphs and interpolating behavior of partial-duality polynomials}
\thanks{This work is supported by the NNSFC (Grant No. 11471106) and  the JSSCRC( Grant No. 2021530). }
\date{\today\copyright}
\author{Qiyao Chen}
\address{College of Mathematics, Hunan University, 410082 Changsha, China}
\email{chen1812020@163.com}
\author{Yichao Chen}
\address{College of Mathematics, Hunan University, 410082 Changsha, China}
\email{chengraph@163.com}
\maketitle

\begin{abstract}
\textwidth=114.3mm
{ Recently, Gross, Mansour and Tucker introduced the partial-twuality polynomials. 
In this paper, we find that when there are enough parallel edges,
any multiple graph  is a negative answer to the problem 8.7 in their paper [European J. Combin. 95 (2021),
103329]: Is the restricted-orientable partial-Petrial polynomial of an arbitrary ribbon graph  even-interpolating$?$  In addition, we also find a counterexample to the conjecture 8.1 of Gross, Mansour and Tucker:
 If the partial-dual genus polynomial is neither an odd nor an even polynomial, then it is interpolating.}
\end{abstract}

\section{Introduction}

 We assume that the readers are familiar with the basic knowledge of topological graph theory.
 The reader is referred to \cite{GMT21a} for the explanation of all terms not defined here.  

 Let  $G^{*|_A}$  ( $G^{\times|_A}$ ) be the \textit{partial dual} (\textit{partial Petrial}) of $G$ with respect to $A \subset E(G)$. Denote by $v(G)$, $e(G)$, $f(G)$ and $c(G)$  the number of \textit{vertices, edges, faces} and \textit{connected components} of $G,$ respectively.  For $\bullet \in\{\times,*, \times*, *\times, *\times*\}$, Gross, Mansour, and Tucker \cite{GMT21a} introduced the \textit{ partial-$\bullet$ polynomial} for the ribbon graph $G$, i.e.,
 $$~^{\partial}{\Eu}^{\bullet}_{G}(z)=\sum\limits_{A\subseteq E(G)}z^{\eu[G^{\bullet|_A}]},$$
where $\eu[G^{\bullet|_A}]$ represents the Euler-genus of $G^{\bullet|_A}$.

They also introduced \textit{the restricted orientable partial-$\times$ polynomial}  of $G$ by enumerating Euler-genus only over edge-subsets $A\subseteq E(G)$
such that $G^{\times|_{A}}$ is orientable. We recall that \textit{partial duality} was introduced by Chmutov in \cite{Chm09}.  In \cite{EM12},  Ellis-Monaghan and Moffatt  extended the partial-duality to include the Wilson dual, the Petrie dual, and the two kinds of triality operators. In \cite{AE19}, Abrams and Ellis-Monaghan  called the five operators twualities. We may refer the
reader to \cite{ EM13, GMT20, GMT21a, GMT21b} for more background about partial-duality and partial Petrial.

{ A \textit{subdivision} of  $G$ is obtained by replacing an  edge $e=uv$ of $G$ by a path $uwv,$ and  a \textit{proper edge} is  an edge with two  different ends.
 The \textit{contraction} on edge $e$ is denoted  by $G/e$, and we denote by $G-e$ the ribbon graph obtained from
$G$ by deleting the ribbon $e$.
  In \cite{GMT20}, Gross, Mantour, and Tucker proved a subdivided edge recursion for partial-$*$.}
\begin{thm} \cite{GMT20}\label{GMT20}
Given a ribbon graph $G$ and a ribbon $e$. Let $K$ be a subdivision of $G$,
then\begin{equation}\label{1027}
~^{\partial}{\Eu}^{*}_{K}(z)=
\begin{cases}
 2~^{\partial}{\Eu}^{*}_{G}(z),&\text{if $e$ is a cut ribbon,  }\\
  ~^{\partial}{\Eu}^{*}_{G}(z)+2z^{2}~^{\partial}{\Eu}^{*}_{G-e}(z),&\text{if $e$ is non-separating.}\\
\end{cases}
\end{equation}
\end{thm}

Gross, Mantour, and Tucker also give  subdivided edges and parallel  edges recursions for partial-$\times$ in \cite{GMT21a}.

\begin{thm}\cite{GMT21a}
Let $G$ be a ribbon graph with an e-type pp ribbon $e$, and let  $G+e'$ be the ribbon graph   obtained  by adding a ribbon $e'$ parallel to a ribbon $e$ in $G$.  Thus
\begin{eqnarray}
~^{\partial}{\Eu}^{\times}_{G+e'}(z) &=&(1+2z) ~^{\partial}{\Eu}^{\times}_{G/e}(z)+(z^{2})~^{\partial}{\Eu}^{\times}_{G-e}(z)
\end{eqnarray}
\end{thm}

\begin{thm}\cite{GMT21a}\label{sub}
Let $G$ be a ribbon graph with a ribbon $e$, and let  $H$ be  obtained  from $G $ by subdividing  $e$ into edges $e_{1}$ and $e_{2}$. Then
\begin{eqnarray}
~^{\partial}{\Eu}^{\times}_{H}(z) &=&2~ ~^{\partial}{\Eu}^{\times}_{G}(z).\label{sub0}
\end{eqnarray}
\end{thm}

 The \textit{support} of a polynomial $f(z)=\dsum_{i=0}^{n}a_iz^i$ is the set $\{i|a_i\neq 0\}$.  If $supp(f(z) )$ is an \textit{integer interval} $[m, n]$
of all  integers from $m$ to $n$, inclusive, we call $f(z)$  an \textit{interpolating polynomial. }The \textit{size} of integer interval $[m, n]$ is the number of elements of $[m, n].$
If $supp(f(z) )$ is the set of all even natural numbers in an integer interval, then $f(z)$ is  \textit{even-interpolating.} If the terms of the non-zero coefficient of the polynomial are even (odd) degree, we call it  an
 \textit{even (odd)  polynomial}.

  Conjecture 8.1 and Problem 8.7 in their paper \cite{GMT21a} state  that

\begin{conj}  If the partial-$*$ polynomial $ ~^{\partial}{\Eu}^{*}_{G}(z)$ is neither an odd nor an even polynomial, then it is interpolating.
\end{conj}

\begin{problem}  Is the restricted-orientable partial-$\times$ polynomial of an arbitrary ribbon graph $G$ even-interpolating$?$
\end{problem}

 In this paper, we apply the operations of adding {parallel edges}  and {subdividing  edges} to get some counter-examples of  Conjecture 8.1 and Problem 1.2 \cite{GMT21a}. We first disprove Conjecture 1.1 by finding an infinite family of ribbon graphs as counterexamples. Then
we answer the Problem 1.2 by proving the restricted orientable partial polynomial of any multiple ribbon graph with enough  parallel edges ( we also have a tight lower bound for this) are not even-interpolating.  

\begin{rem} Throughout the paper, we will use the rotation projection \cite{GT87} of $G$ instead of the ribbon graph $G$ itself (See Figure \ref{flm} - Figure \ref{f111}).
\end{rem}
\section{The parallel edge recursion for partial- $*$
polynomials  }

In this section,  we  derive a recursion for the partial-$*$
polynomials.
 Given two disjoint ribbon graphs $G_{1}$ and $G_{2}$, we let $G_{1}\vee G_{2}$ denote
the \textit{join} of $G_{1}$ and $G_{2}$.    The \textit{complement} of $A$ in $E(G)$ is  $A^{c}=E(G)-A$.

\begin{thm}\cite{Mof12}\label{d}
Let $G$ be a ribbon graph and $A\subseteq E(G)$, then
\begin{eqnarray}\label{d0}
\eu(G^{*|_{ A}})&=&2c(G)+e(G)-f(A)-f(A^{c}).
\end{eqnarray}

\end{thm}

\begin{thm}\label{1011}
Let $G$ be a ribbon graph with a  proper ribbon $e$,  let $G+e_{1}$ be the ribbon graph obtained by adding  parallel ribbon $e_1$ to the ribbon $e$  as shown in Figure \ref{flm}. Then
\begin{eqnarray}\label{dm0}
~^{\partial}{\Eu}^{*}_{G+ e_{1}}(z)&=&~^{\partial}{\Eu}^{*}_{G}(z)
+2z^{2}~^{\partial}{\Eu}^{*}_{G/ e}(z).
\end{eqnarray}
\end{thm}
\begin{figure}[h]
  \begin{minipage}[t]{0.6\textwidth}
  \centering
  \includegraphics[width=1\textwidth]{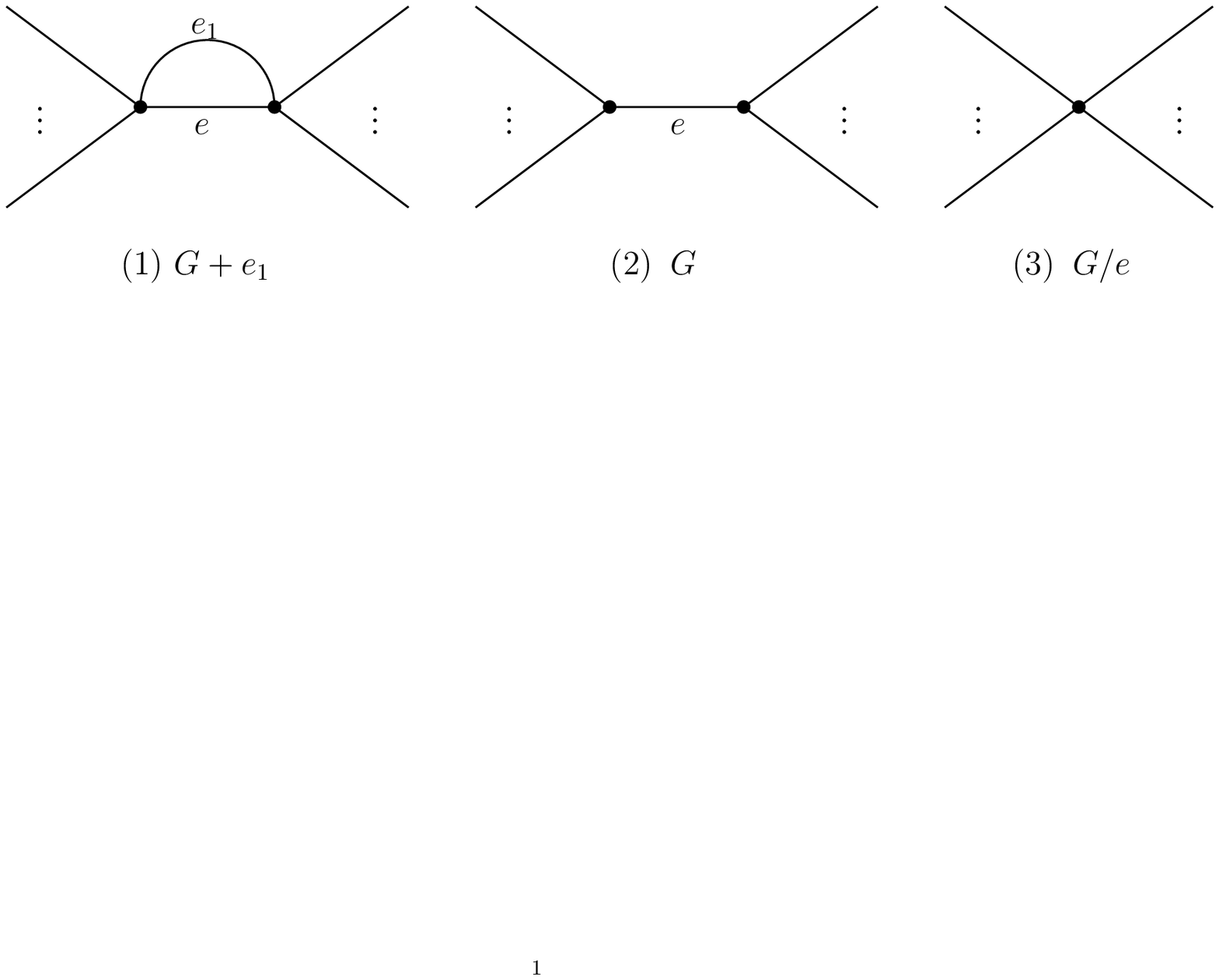}

  \end{minipage}

  \caption{}
   \label{flm}
  \end{figure}

\begin{proof}

Let $A\subseteq E(G+e_{1}),\ A_{1}\subseteq E(G), \  A_{2}\subseteq E(G/ e).$ It is easy to see that $e(G+e_{1})=e(G/ e)+2=e(G)+1.$     There are three cases.

$\mathbf{Case~ 1}:$ Suppose that $e_{1}\in A$ and $e\in A,$ and let $A_{1}=A-e_{1}$.   Since the ribbon graph $A_{1}$ is  obtained  from  $A$ by  deleting the multiple edge $e_{1}$, which decreases the number of faces by 1,  we have
\begin{eqnarray}\label{A1}
f(A)=f(A_{1})+1.
\end{eqnarray}  We use the fact that  the complement of $A$ in  $E(G+e_{1})$ is equal to the complement of $A_{1}$ in $E(G),$ then
\begin{eqnarray}\label{A2}
f(A^{c})=f(A_{1}^{c}).
\end{eqnarray}
 Thus, by  Theorem \ref{d},
\begin{eqnarray}\label{dm1}
\eu((G+e_{1})^{*|_{ A}})&=&2c(G+e_{1})+e(G+e_{1})-f(A)-f(A^{c})\notag\\
&=&2c(G)+e(G)+1-f(A_{1})-1-f(A_{1}^{c})\quad \text{by (\ref{A1})~and~(\ref{A2}) }\notag\\
&=&\eu(G^{*|_{ A_{1}}}).
\end{eqnarray}

$\mathbf{Case ~2}:$
Assume that $e_{1},e\in A^{c},$  and  let $A=A_{1}$.  Obviously,
\begin{eqnarray}\label{A3}
f(A)=f(A_{1}).
\end{eqnarray}
Because  $A^{c}_{1}$  is  obtained  from  $A^{c}$ by  deleting the multiple edge $e_{1}$, which decreases the number of faces by 1. Therefore,
\begin{eqnarray}\label{A4}
f(A^{c})=f(A^{c}_{1})+1.
\end{eqnarray}   Hence
 \begin{eqnarray}\label{dm2}
\eu((G+e_{1})^{*|_{ A}})&=&2c(G+e_{1})+e(G+e_{1})-f(A)-f(A^{c})\quad\text{by (\ref{d0}) }\notag\\
&=&2c(G)+e(G)+1-f(A_{1})-f(A_{1}^{c})-1\quad\text{by (\ref{A3})~and~(\ref{A4}) }\notag\\
&=&\eu(G^{*|_{ A_{1}}}).
\end{eqnarray}
 Let $\mathscr{A}=\{A|e_{1},e\in A\}\cup \{A|e_{1},e\in A^{c}\},$ $\mathscr{A}_{1}=\{A_{1}|e\in A_{1}\},$ $\mathscr{A}_{2}=\{A_{1}|e\in A^{c}_{1}\},$ then
 \begin{eqnarray}\label{dm3}
\sum\limits_{A\in\mathscr{A}}z^{\eu\big((G+e_{1})^{*|_A}\big)} &=&\sum\limits_{A_{1}\in \mathscr{A}_{1}} z^{\eu\big(G^{*|_{ A_{1}}}\big)}+\sum\limits_{A_{1}\in \mathscr{A}_{2}} z^{\eu\big(G^{*|_{ A_{1}}}\big)}\quad\text{by (\ref{dm1})~and~(\ref{dm2}) }\notag\\
&=&~^{\partial}{\Eu}^{*}_{G}(z).
\end{eqnarray}

$\mathbf{Case ~3}:$
Let $e_{1}\in A^{c},$ $e\in A$, and we let $A_{2}=A/e$.  Recall that $ A_{2}$ is  obtained  from $ A $ by contracting the proper edge  $e$ and  $A^{c}_{2}$  is  obtained  from  $A^{c}$  by contracting the proper edge  $e_{1}$.  Since the contraction  does not change the number of faces, we have
\begin{eqnarray}\label{A5}
f(A)=f(A_{2}), \quad f(A^{c})=f(A_{2}^{c}).
\end{eqnarray}
    Therefore,
\begin{eqnarray}\label{dm4}
\eu((G+e_{1})^{*|_{ A}})&=&2c(G+e_{1})+e(G+e_{1})-f(A)-f(A^{c})\quad\text{by (\ref{d0}) }\notag\\
&=&2c(G/e)+e(G/e)+2-f(A_{2})-f(A_{2}^{c})\quad\text{by (\ref{A5})  }\notag\\
&=&\eu((G/e)^{*| _{A_{2}}})+2.
\end{eqnarray}

 Similarly, formula (\ref{dm4}) also holds for the case: $e\in A^{c},$ and $e_{1}\in A.$

 Let  $\overline{\mathscr{A}}=\{A|e_{1}\in A^{c}, e\in A\}\cup \{A|e\in A^{c}, e_{1}\in A\},$ and  then
 \begin{eqnarray}\label{dm5}
\sum\limits_{A\in \overline{\mathscr{A}}}z^{\eu\big((G+e_{1})^{*|_A}\big)}&=&2z^{2} ~^{\partial}{\Eu}^{*}_{G/e}(z)  \quad\text{by (\ref{dm4}) }
\end{eqnarray}
Combining  cases 1-3, we obtain
$$
\begin {aligned}
 ~^{\partial}{\Eu}^{*}_{G+e_{1}}(z)&=\sum\limits_{A\in\mathscr{A}} z^{\eu[(G+e_{1})^{*|_A}]}+\sum\limits_{A\in\overline{\mathscr{A}}} z^{\eu[(G+e_{1})^{*|_A}]}\\
 &=2z^{2}~^{\partial}{\Eu}^{*}_{G/ e}(z)+~^{\partial}{\Eu}^{*}_{G}(z).\\
\end{aligned}
$$

\end{proof}

\begin{figure}[h]
  \centering
  \includegraphics[width=0.5\textwidth]{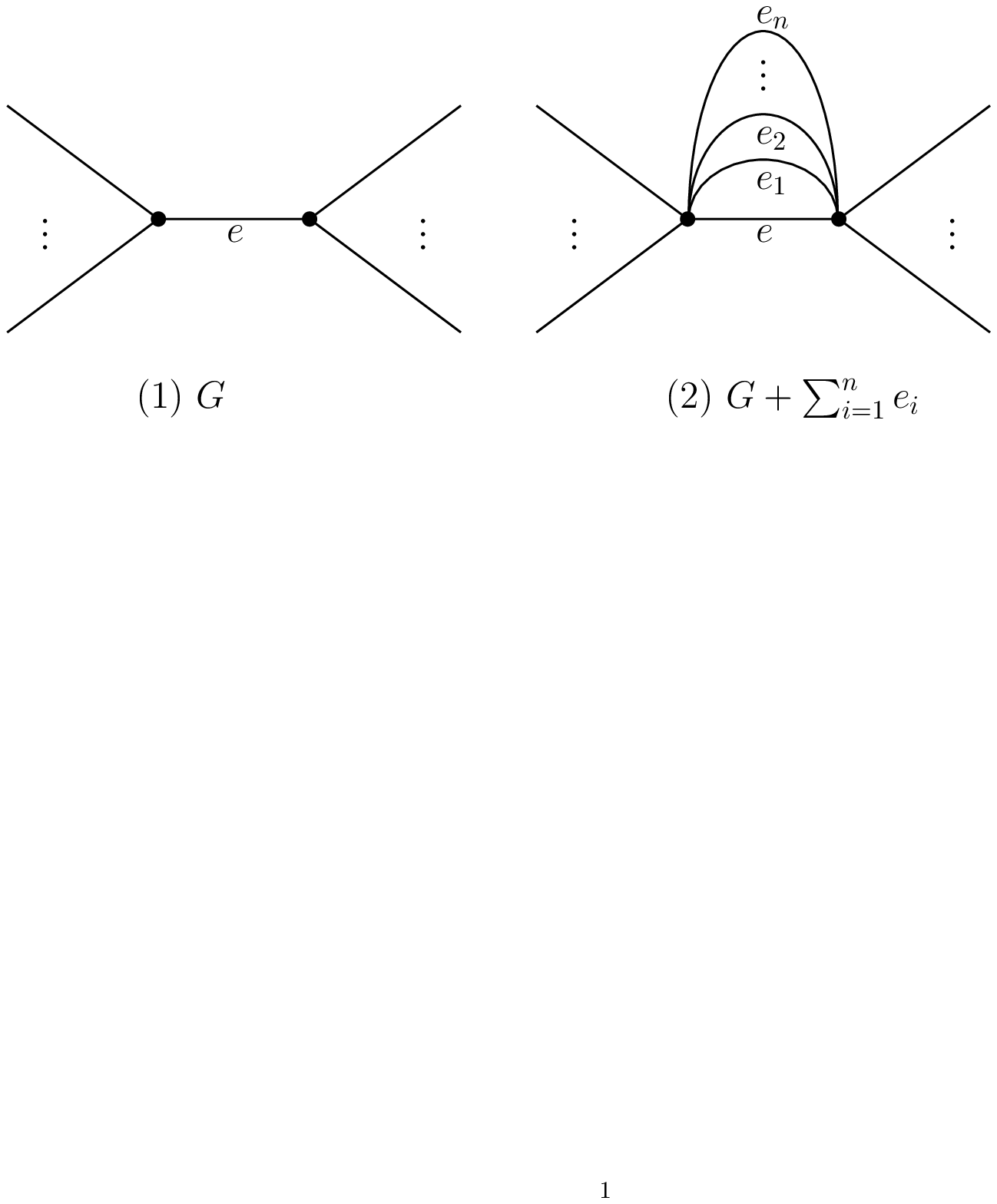}
   \caption{ }
 \label{dm0522}
\end{figure}
\begin{cor}\label{dmt}
Let $G$ be a ribbon graph with a  proper ribbon $e$, and let $G+\sum ^{n}_{i=1}e_{i}$ be the ribbon graph obtained by adding $n$ parallel ribbons $e_1$, $e_2$, $\cdots$,  $e_n$ to the ribbon $e$  as shown in Figure \ref{dm0522}.  Then
\begin{eqnarray}\label{dm01}
~^{\partial}{\Eu}^{*}_{G+ \sum ^{n}_{i=1}e_{i}}(z)&=&~^{\partial}{\Eu}^{*}_{G}(z)
+(2^{n+1}-2)z^{2}~^{\partial}{\Eu}^{*}_{G/ e}(z).
\end{eqnarray}
\end{cor}
\begin{proof}
By Theorem \ref{1011},
\begin{eqnarray}\label{dm6}
~^{\partial}{\Eu}^{*}_{G+ \sum ^{n}_{i=1}e_{i}}(z)&=&~^{\partial}{\Eu}^{*}_{G+ \sum ^{n-1}_{i=1}e_{i}}(z)
+2z^{2}~^{\partial}{\Eu}^{*}_{(G+ \sum ^{n-1}_{i=1}e_{i})/ e_{n-1}}(z).
\end{eqnarray}
 Recall that the ribbons $e$, $e_1$,  $e_2$, $\cdots$, $e_{n-2}$  in $(G+ \sum ^{n-1}_{i=1}e_{i})/ e_{n-1}$ are all loops, and the ribbon graph
$G/ e$ is obtained from $(G+ \sum ^{n-1}_{i=1}e_{i})/ e_{n-1}$ by deleting $n-1$ loops. Thus,
from Proposition 3.2 in  \cite{GMT20},
\begin{eqnarray}\label{dm7}
~^{\partial}{\Eu}^{*}_{(G+ \sum ^{n-1}_{i=1}e_{i})/ e_{n-1}}(z)&=&2^{n-1}~^{\partial}{\Eu}^{*}_{G/ e}(z).
\end{eqnarray}
(\ref{dm6}), (\ref{dm7}) and (\ref{dm0}) clearly imply (\ref{dm01}).

\end{proof}

Recall that the  Tutte polynomial (\cite{Tut54}) and Bollobas-Riordan polynomial (\cite{Bol02})  have  deletion-contraction  recursions, as well as the role of series-parallel graphs in the foundations of matroids(\cite{Bry71}).

\begin{rem}As pointed out by the anonymous referee, subdividing an edge is adding a parallel edge in the dual ribbon graph.
Thus Theorem \ref{1011} can be seen as the dual form of Theorem \ref{GMT20}.
\end{rem}

Let $C_{n}$ and $K_{n}$ denote a $n$-cycle and a complete graph with $n$ vertices, respectively.  A graph is \textit{series parallel} (\cite{AVJ99}), if its 2-connected components can be generated by repeatedly  adding a parallel edge and subdividing an edge, starting with $K_{2}$. As an example, we apply Theorem \ref{GMT20} and  Theorem \ref{1011} to calculate the partial -$*$ polynomial for a series parallel graph.

\begin{figure}[h]
  \centering
  \includegraphics[width=0.6\textwidth]{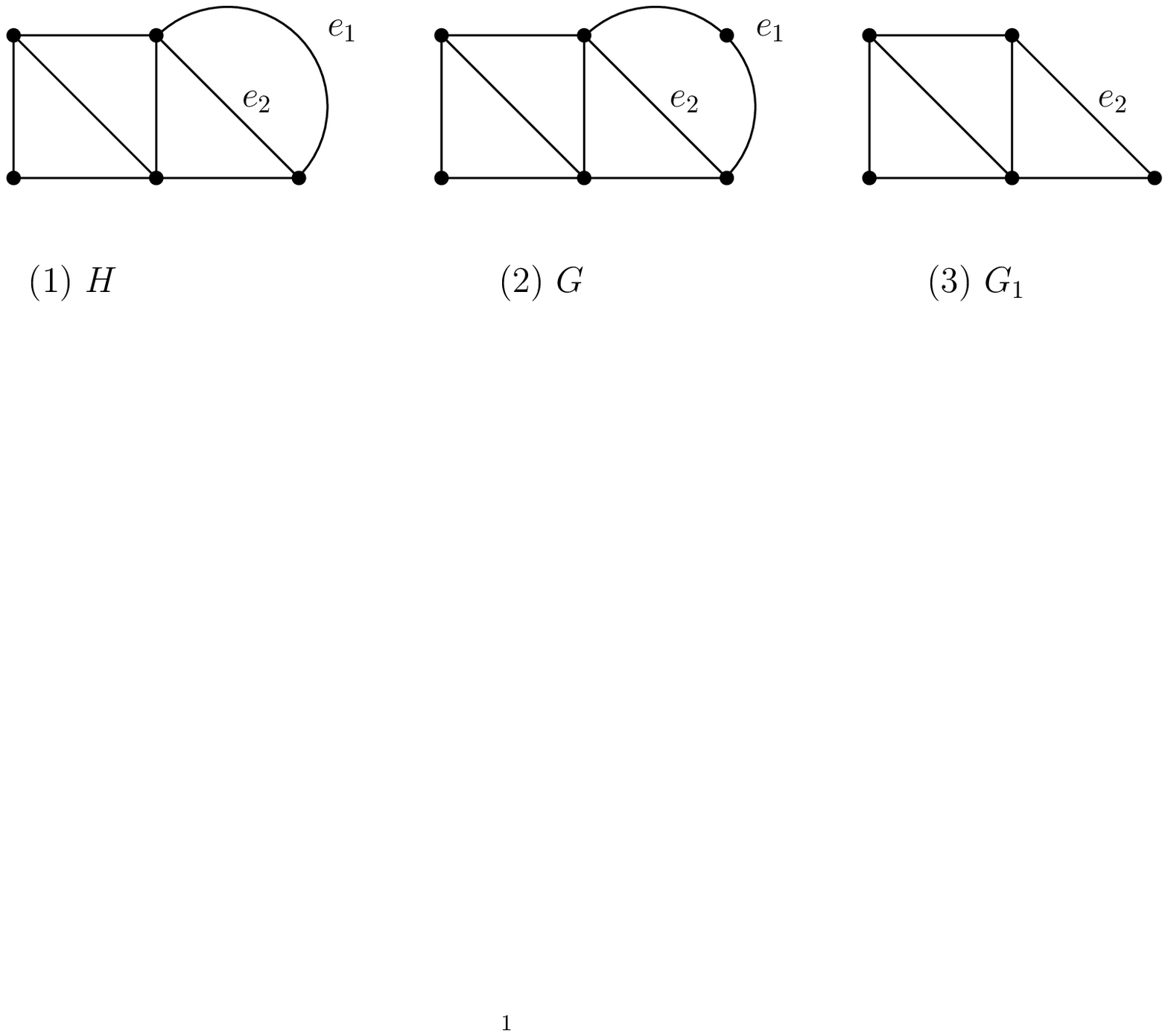}
   \caption{ }
 \label{pn1027}
\end{figure}
\begin{figure}[h]
  \centering
  \includegraphics[width=0.6\textwidth]{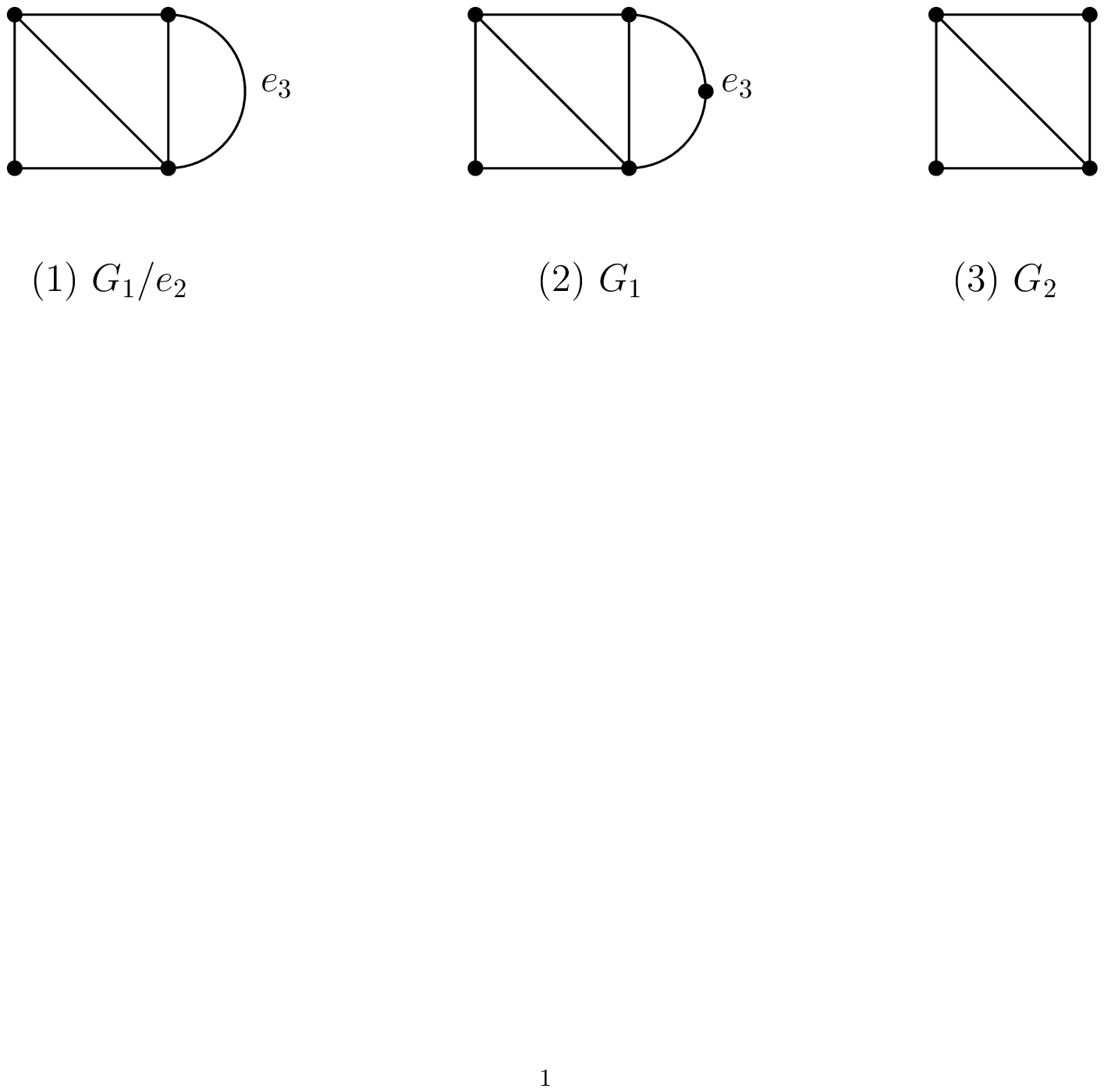}
   \caption{ }
 \label{pn10271}
\end{figure}
\begin{example}
  Figure \ref{pn1027} shows the  series parallel graphs $H, G$ and $G_1.$ 
It's clear that $G$ is isomorphic to $H$ with ribbon $e_1$ subdivided once, and  $G_{1}=H-e_{1}$. 
Moreover, $H$ is obtained from $G_1$ by adding a parallel edge $e_{1}$ to the ribbon $e_2$, hence, we have
\begin{eqnarray}\label{2400}
~^{\partial}{\Eu}^{*}_{G}(z)&=&~^{\partial}{\Eu}^{*}_{H}(z)+2z^{2}~^{\partial}{\Eu}^{*}_{G_{1}}(z)
\quad\text{by~(\ref{1027})} \notag\\
&=&(2z^{2}+1)~^{\partial}{\Eu}^{*}_{G_{1}}(z)+2z^{2}~^{\partial}{\Eu}^{*}_{G_{1}/e_{2}}(z)\quad\text{by ~ (\ref{dm0})}
\end{eqnarray}
Since $G_1$ is isomorphic to $G_{1}/e_{2}$ with ribbon $e_3$ subdivided once, let $G_{2}=(G_{1}/e_{2})-e_{3}$ as in Figure \ref{pn10271}, it follows that
\begin{eqnarray}\label{2401}
~^{\partial}{\Eu}^{*}_{G_{1}}(z)
&=&~^{\partial}{\Eu}^{*}_{G_{1}/e_{2}}(z)+2z^{2}~^{\partial}{\Eu}^{*}_{G_{2}}(z)\quad\text{by ~ (\ref{1027})}
\end{eqnarray}
Combining (\ref{2400}) and (\ref{2401}), we have
\begin{eqnarray}\label{2402}
~^{\partial}{\Eu}^{*}_{G}(z)&=& (4z^{2}+1)~^{\partial}{\Eu}^{*}_{G_{1}}(z)-4z^{4}~^{\partial}{\Eu}^{*}_{G_{2}}(z).
\end{eqnarray}
Similarly, we have
\begin{eqnarray}\label{2403}
~^{\partial}{\Eu}^{*}_{G_{1}}(z)&=& (4z^{2}+1)~^{\partial}{\Eu}^{*}_{G_{2}}(z)-4z^{4}~^{\partial}{\Eu}^{*}_{K_{3}}(z)\quad\text{by ~ (\ref{2402})},
\end{eqnarray}
and
\begin{eqnarray}\label{2404}
~^{\partial}{\Eu}^{*}_{G_{2}}(z)&=& (4z^{2}+1)~^{\partial}{\Eu}^{*}_{K_{3}}(z)-4z^{4}~^{\partial}{\Eu}^{*}_{K_{2}}(z)\quad\text{by ~ (\ref{2402})},
\end{eqnarray}
where  $K_{3}$ is obtained by adding a parallel edge  and subdividing an edge to $K_{2}$ once. It's obvious that $~^{\partial}{\Eu}^{*}_{K_{3}}(z)=6z^{2}+2$, and $~^{\partial}{\Eu}^{*}_{K_{2}}(z)=2.$
Therefore,
combining (\ref{2402}), (\ref{2403}) and (\ref{2404}),  we have
\begin{eqnarray*}\label{2405}
~^{\partial}{\Eu}^{*}_{G}(z)&=& (32z^{6}+40z^{4}+12z^{2}+1)~^{\partial}{\Eu}^{*}_{K_{3}}(z)-(48z^{8}+332z^{6}+4z^{4})~^{\partial}{\Eu}^{*}_{K_{2}}(z)
\notag\\
&=& 96z^{8}+240z^{6}+144z^{4}+30z^{2}+2.
\end{eqnarray*}

\end{example}

\section{Some counterexamples to the
conjecture 1.1.}

{  In this section, we find some infinite families of ribbon graphs as  counterexamples to  Conjecture 1.1.  }
\begin{figure}[h]
  \centering
  \includegraphics[width=0.7\textwidth]{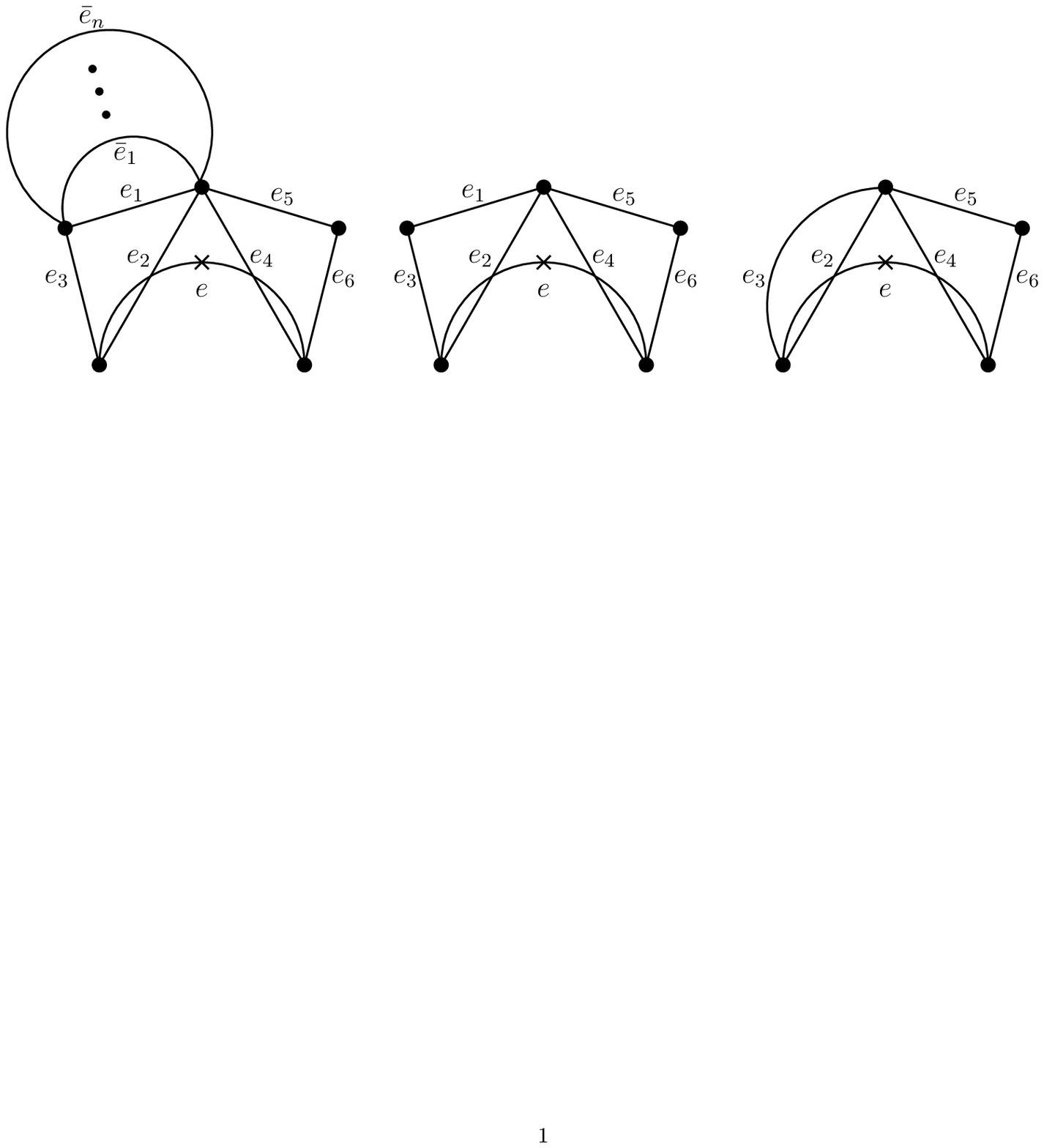}
 \caption{$G+\sum ^{n}_{i=1}\bar{e}_{i}$ (left), $G$ (middle) and $G/e_{1}$ (right) }
 \label{m0525}
\end{figure}

  Let $G$ and $G+\sum ^{n}_{i=1}\bar{e}_{i}$ be the ribbon graphs of Figure \ref{m0525}.
 Clearly,  $e(A)+e(A^{c})=e(G)$ and $\eu(G^{*|_{ A}})=\eu(G^{*|_{ A^{c}}})$. { Note that $e(G)=7$, and for each edge set with the number of edges greater than 3 in $G$, we can find the complement  with the number of edges less than 4.}

  \begin{enumerate}
    \item
   $G$
  has partial-$*$ polynomial $8z^{2}+48z^{4}+32z^{5}+40z^{6}$. By the previous analysis, we need to show the partial-$*$ polynomial of dualizing edge set with the number of edges less than 4 is $4z^{2}+24z^{4}+16z^{5}+20z^{6}$.
  The
$z^{2}$ term corresponds to dualizing none, or twisted edge $e$,  or all edges in the
same $C_{3}$ (2 choices) and all are isomorphic to $G$. The $z^{4}$
terms comes from dualizing one
edge in the same $C_{3}$ (6 choices),  or a pair in the same $C_{3}$ (6 choices),  or one edge in the same $C_{3}$ and $e$  (6 choices),  or a pair in the same $C_{3}$ and $e$ (6 choices). Dualizing  one edge of $C_{3}$ and  two edges of the other $C_{3}$ produces $z^{5}$ and $z^{6}$ (18 choices). Where dualizing $e_{1}$, $e_{3}$, and  one edge of $C_{3}$ ($e_{4}$, $e_{5}$,  $e_{6}$ ) (3 choices), or $e_{5}$, $e_{6}$, and  one edge of $C_{3}$ ($e_{1}$, $e_{2}$,  $e_{3}$ ) (3 choices)  produces $z^{5}$,  the remaining 12 choices produces $z^{6}$.  Dualizing  one edge in each $C_{3}$ and $e$ produces $z^{5}$ and $z^{6}$ (9 choices). Where dualizing $e_{2}$, $e_{4}$, and  $e$   produces $z^{5}$,  the remaining 8 choices produces $z^{6}$. Dualizing  one edge in each $C_{3}$ (9 choices) also produces $z^{5}$.
    \item  $ G/e_{1}$ has partial-$*$ polynomial $8z^{6}+16z^{5}+32z^{4}+8z^{2}.$ For each edge set with the number of edges greater than 3 in $G/e_{1},$ we can find the complement with the number of edges less than 3. Moreover, for the edge set whose number of edges is 3, the edge set with twisted edge $e$ is the complement of the edge set without $e$. Therefore, we need to show the partial-$*$ polynomial of dualizing edge set with the number of edges not greater than 2, and  edge set without $e$ where the number of edges is 3 are $4z^{6}+8z^{5}+16z^{4}+4z^{2}$.
  The $z^{2}$
terms comes from dualizing none, or $e$,   or $C_{3}$, or $C_{2}$ and all are isomorphic to $G/e_{1}$. The $z^{4}$
terms comes from dualizing one edge in $C_{3}$ and $C_{2}$   (5 choices),   or a pair in  $C_{3}$ (3 choices), or $e$ and one edge in $C_{3}$ and $C_{2}$ (5 choices),
or all edges in  $C_{2}$ and one edge in $C_{3}$ (3 choices). Dualizing  one edge of $C_{3}$ and  one edge of $C_{2}$  produces $z^{5}$ and $z^{6}$ (6 choices). Where dualizing $e_{3}$, $e_{5}$, $e_{6}$, and  $e_{2}$, $e_{5}$, $e_{6}$ produces $z^{5}$,  the remaining 4 choices produces $z^{6}$. Dualizing  one edge in  $C_{3}$ and one edge in $C_{2}$  (6 choices) also produces $z^{5}$.

  \end{enumerate}

   Thus by Corollary \ref{dmt}, we have
 \begin{eqnarray*}
~^{\partial}{\Eu}^{*}_{G+ \sum ^{n}_{i=1}\bar{e}_{i}}(z)&=&~^{\partial}{\Eu}^{*}_{G}(z)
+(2^{n+1}-2)z^{2}~^{\partial}{\Eu}^{*}_{G/ e_{1}}(z)\notag\\
&=&(2^{n+4}-16)z^{8}+(2^{n+5}-32)z^{7}+(2^{n+6}-24)z^{6}\\ &&+ 32z^{5}+(2^{n+4}+32)z^{4}+8z^{2}.
\end{eqnarray*}

It is quite obvious that the polynomial above is neither an odd nor an
even polynomial,  it's also not interpolating.

 \begin{rem}The anonymous referee also pointed out that we can take lots of joins of $G$ with itself to get counterexamples, since $$(z^{6}+z^{5}+z^{4}+z^{2})^{n}=z^{2n}+nz^{2n+2}+nz^{2n+3}+\cdots.$$
\end{rem}


\section{ A solution  to the  Problem 1.2.}
{{ In this section, we discuss the effect
of adding multiple edges on the restricted orientable
 partial-$\times$ polynomial. The \textit{restricted orientable partial-$\times$ ribbon graph} of $G$ is the orientable ribbon graph $G^{\times|_{A}}$.
For short, we
 denote the restricted orientable partial-$\times$ polynomial of ribbon graph $G$ by  $~^{\partial}{\Eu}^{\times}_{G}|_O(z)$.}}

 \begin{example}\label{dnn}
{Let $D_{n}$ be the dipole  ribbon graph in the sphere.  Clearly, if
one twists any proper subset of the edges, the resulting ribbon graph is non orientable, there will be circuits of lenght two containing only one twisted
edge. Thus the only orientable partial -$\times$ duals are $D_{n}$ and $D_{n}^{\times}$. The latter
has $1$ or $2$ faces depending on whether $n$ is odd or even and hence has
orientable genus $\frac{n-1}{2}$ or $\frac{n-2}{2}$. Furthermore, we have $~^{\partial}{\Eu}^{\times}_{D_{n}}|_O(z)=1+z^{n-1}$, when $n$ is odd;
$~^{\partial}{\Eu}^{\times}_{D_{n}}|_O(z)=1+z^{n-2}$, when $n$ is even.
Thus,   when $n\geq 5,$  the restricted-orientable partial-$\times$ polynomial of $D_{n}$ is not even-interpolating.}

\end{example}

\begin{lemma}\label{ori}
 Let $G$ be a  ribbon graph, then
$~^{\partial}{\Eu}^{\times}_{G}|_O(z)=~^{\partial}{\Eu}^{\times}_{G^{\times|_{A}}}|_O(z).$
\end{lemma}

\begin{proof}
Let $A_{1}\subseteq E(G)$, and $A\subseteq E(G)$.
For each  orientable ribbon graph $G^{\times|_{A_{1}}}$,  there exists  $B\subseteq E(G^{\times|_{A}})$ such that the  orientable ribbon graph
   $(G^{\times|_{A}})^{\times|_{B}}$  and $G^{\times|_{A_{1}}}$  are the same.

\end{proof}

 Lemma \ref{ori} shows that   it suffices to consider the case that  $G$ is an orientable ribbon graph.

\begin{lemma}\label{num}
Let $G$ be a  ribbon graph, then the number of the restricted orientable partial-$\times$ ribbon graphs of $G$ is $2^{v(G)-1}$.
\end{lemma}
\begin{proof}
By \cite{GMT21a}, we know  that  the proportion of partial-$\times$ duals of G
that are orientable is $\frac{1}{2^{\beta(G)}}$
.  There are $2^{e(G)}$ spanning subgraphs in $G$, thus, the number of the restricted orientable partial-$\times$ ribbon graphs of $G$ is
\begin{eqnarray*}
\frac{ 2^{e(G)}}{2^{\beta(G)}} &=& \frac{2^{e(G)}}{2^{e(G)-v(G)+1}}
=2^{v(G)-1}.\label{2vG-1}
\end{eqnarray*}

\end{proof}

Let $e$ be a ribbon of $G$, and let $A$ be a subset of $E(G)$,  we  define  $\eu^{0}(G^{\times|_{A}})$ and $\eu^{1}(G^{\times|_{A}})$ as $e\in A$ and  $e\notin A$ of the Euler-genus of orientable ribbon graph $G^{\times|_{A}}$, respectively. Similarly,
 we let  $f^{0}(G^{\times|_{A}})$ and $f^{1}(G^{\times|_{A}})$, respectively,  denote the number of     faces of  orientable ribbon graph $G^{\times|_{A}}$ with $e\in A$ and  $e\notin A$. And let
 $ f^{0}_{max}(G)=max\{f^{0}(G^{\times|_{A}})| e\in A\}$,
 $ f^{1}_{min}(G)=min\{f^{1}(G^{\times|_{A}})| e\notin A\}$.

\begin{thm}\label{tl}
Let $G$ be a ribbon graph with a  proper ribbon $e$, and for even $n$, let $G+\sum ^{n}_{i=1}e_{i}$ be the ribbon graph obtained by adding $n$  parallel ribbons $e_1$, $e_2$, $\cdots$,  $e_n$ to the ribbon $e$  (see Figure \ref{dm0522}). Then the polynomial $~^{\partial}{\Eu}^{\times}_{G+\sum ^{n}_{i=1}e_{i}}|_O(z)$ is not  even-interpolating for   sufficiently large $n$.

\end{thm}

\begin{proof}
 Suppose that $~^{\partial}{\Eu}^{\times}_{G+\sum ^{n}_{i=1}e_{i}}|_O(z)$  is even-interpolating. We  need  to analysis the number of     faces for orientable ribbon graphs $G^{\times|_{A}}$ and $(G+\sum ^{n}_{i=1}e_{i})^{\times|_{A'}}$.

 If $e\notin A$,  and $G^{\times|_{A}}$ is an orientable ribbon graph, then $f(G^{\times|_{A}})=f^{1}(G^{\times|_{A}})$.  We let $A'=A$, thus $(G+\sum ^{n}_{i=1}e_{i})^{\times|_{A'}}$ is also an orientable ribbon graph,  and
 \begin{eqnarray}
f((G+\sum ^{n}_{i=1}e_{i})^{\times|_{A'}})&=&f^{1}((G+\sum ^{n}_{i=1}e_{i})^{\times|_{A'}})\notag\\
 &=&f^{1}(G^{\times|_{A}})+n,\label{A.0}
\end{eqnarray}
   because $G^{\times|_{A}}$ is  obtained  from $ (G+\sum ^{n}_{i=1}e_{i})^{\times|_{A'}} $ by deleting  $n$ multiple ribbons $e_{1}$, $e_{2}$, $\cdots$,  $e_{n}$, and deleting a multiple ribbon will reduce one face.

Otherwise, put $e\in A_{0}$,  suppose that the ribbon graph $G^{\times|_{A_{0}}}$ is an orientable  ribbon graph, then  $f(G^{\times|_{A_{0}}})=f^{0}(G^{\times|_{A_{0}}})$.   Due to $e_i \cup e$ is a 2-cycle in $G+\sum ^{n}_{i=1}e_{i}$,  let $A'_{0}=A_{0}\cup\sum ^{n}_{i=1}e_{i}$,   thus $ (G+\sum ^{n}_{i=1}e_{i})^{\times|_{A'_{0}}}$ is also an orientable ribbon graph.  Since the ribbon graph $G^{\times|_{A_{0}}}$ is  obtained  from  $(G+\sum ^{n}_{i=1}e_{i})^{\times|_{A'_{0}}}$  by deleting $n$ twisted multiple ribbons  $e_{1}$, $e_{2}$, $\cdots$,  $e_{n}$, and the deletion of the even number of  twisted ribbons   does not change the number of faces, it follows that
 \begin{eqnarray}
f((G+\sum ^{n}_{i=1}e_{i})^{\times|_{A'_{0}}}) &=&f^{0}((G+\sum ^{n}_{i=1}e_{i})^{\times|_{A'_{0}}})\notag\\
&=& f^{0}(G^{\times|_{A_{0}}}).\label{A.1}
\end{eqnarray}

According to the relationship between $f^{1}(G^{\times|_{A}})$ and $f^{0}(G^{\times|_{A_{0}}}) $, we have the following two cases.

$\mathbf{Case~ 1}:$
 If $f^{1}(G^{\times|_{A}})\geq f^{0}(G^{\times|_{A_{0}}}) $.    By Euler's formula,
  \begin{eqnarray}
   \eu^{0}(G^{\times|_{A_{0}}})&=&2c(G)+e(G)-v(G)- f^{0}(G^{\times|_{A_{0}}})\notag\\
   &\geq& 2c(G)+e(G)-v(G)- f^{1}(G^{\times|_{A}})\notag\\
    &=& \eu^{1}(G^{\times|_{A}}),\label{A.3}
    \end{eqnarray}
    \begin{eqnarray}
    \eu^{0}((G+\sum ^{n}_{i=1}e_{i})^{\times|_{A'_{0}}})
    &=&2c(G)+e(G)+n-v(G)- f^{0}(G^{\times|_{A_{0}}})\quad\text{by (\ref{A.1})}\notag\\
    &=&\eu^{0}(G^{\times|_{A_{0}}})+n,\label{A.5}
     \end{eqnarray}
   \begin{eqnarray}\eu^{1}((G+\sum ^{n}_{i=1}e_{i})^{\times|_{A'}})&=&2c(G)+e(G)+n-v(G)- f^{1}(G^{\times|_{A}})-n\quad\text{by (\ref{A.0})}\notag\\
   &=&\eu^{1}(G^{\times|_{A}}).\label{A.6} \end{eqnarray}

   Moreover,  by (\ref{A.3})-(\ref{A.6}), we get $\eu^{0}((G+\sum ^{n}_{i=1}e_{i})^{\times|_{A'_{0}}}) \geq \eu^{1}((G+\sum ^{n}_{i=1}e_{i})^{\times|_{A'}})$.  So by the hypothesis that $~^{\partial}{\Eu}^{\times}_{G+\sum ^{n}_{i=1}e_{i}}|_O(z)$  is even-interpolating,
   there is a subinterval
   $[\eu^{1}((G+\sum ^{n}_{i=1}e_{i})^{\times|_{A'}}),
   \eu^{0}((G+\sum ^{n}_{i=1}e_{i})^{\times|_{A'_{0}}})]$   in $supp(~^{\partial}{\Eu}^{\times}_{G+\sum ^{n}_{i=1}e_{i}}|_O(z))$. Since there are only even numbers in  $[\eu^{1}((G+\sum ^{n}_{i=1}e_{i})^{\times|_{A'}}),\eu^{0}((G+\sum ^{n}_{i=1}e_{i})^{\times|_{A'_{0}}})]$ (the Euler-genus of an orientable ribbon graph is even), it follows that the size of $[\eu^{1}((G+\sum ^{n}_{i=1}e_{i})^{\times|_{A'}}),
   \eu^{0}((G+\sum ^{n}_{i=1}e_{i})^{\times|_{A'_{0}}})]$ is \begin{eqnarray}\label{05291}
 \frac{\eu^{0}(G^{\times|_{A_{0}}})-\eu^{1}(G^{\times|_{A}})+n}{2}+1 \quad\text{by (\ref{A.5}) and (\ref{A.6}) }
      \end{eqnarray}

       By Lemma \ref{num},
  there are $2^{v(G)-1}$ orientable ribbon graphs   in $(G+\sum ^{n}_{i=1}e_{i})^{\times|_{A'}}$.  Even if the number of faces of  the restricted orientable partial-$\times$ ribbon graph of
  $(G+\sum ^{n}_{i=1}e_{i})$ are different, there are only $2^{v(G)-1}$ different numbers.
By (\ref{05291}), for  sufficiently large $n$,  we have
   $$\frac{\eu^{0}(G^{\times|_{A_{0}}})-\eu^{1}(G^{\times|_{A}})+n}{2}+1\geq 2^{v(G)-1}+1.$$
 Hence, contradicting the assumption that  $~^{\partial}{\Eu}^{\times}_{G+\sum ^{n}_{i=1}e_{i}}|_O(z)$  is   even-interpolating.
 We have proved the case~ 1.

$\mathbf{Case~ 2}:$ {If $f^{1}(G^{\times|_{A}})< f^{0}(G^{\times|_{A_{0}}}) $. By Euler's formula, we have $\eu^{0}(G^{\times|_{A_{0}}})<\eu^{1}(G^{\times|_{A}})$,
however, when $n$ is large enough, by (\ref{A.5})-(\ref{A.6}), we still obtain $\eu^{0}((G+\sum ^{n}_{i=1}e_{i})^{\times|_{A'_{0}}}) \geq \eu^{1}((G+\sum ^{n}_{i=1}e_{i})^{\times|_{A'}})$.
 The remainder of the proof is analogous to that in case~ 1 and so is omitted.}

\end{proof}

 Next, we give a   lower bound for    $~^{\partial}{\Eu}^{\times}_{G+\sum ^{n}_{i=1}e_{i}}|_O(z)$  is not  even-interpolating for the number $n$.
\begin{thm}\label{tm}
Let $G$ be a ribbon graph with a  proper ribbon $e$, and let $G+\sum ^{n}_{i=1}e_{i}$ be the ribbon graph obtained by adding $n$ parallel ribbons $e_1$, $e_2$, $\cdots$,  $e_n$ to the ribbon $e$ (see Figure \ref{dm0522}).
\begin{enumerate}
 \item If $~^{\partial}{\Eu}^{\times}_{G}|_O(z)$ has only one term, and \begin{equation*}
n\geq
\begin{cases}
 4,&\text{if $f^{0}_{max}(G +e_{1})= f^{0}(G^{\times|_{A}})+1$,  }\\
 3,&\text{if $f^{0}_{max}(G +e_{1})= f^{0}(G^{\times|_{A}})-1$,}\\

\end{cases}
\end{equation*} then  $~^{\partial}{\Eu}^{\times}_{G+\sum ^{n}_{i=1}e_{i}}|_O(z)$  is not  even-interpolating.
\item
    If $~^{\partial}{\Eu}^{\times}_{G}|_O(z)$ has more than one term, and $
    n\geq min\{f^{0}_{max}(G +e_{1})- f^{1}_{min}(G)+4, f^{0}_{max}(G)-f^{1}_{min}(G)+4\},
   $
     then $~^{\partial}{\Eu}^{\times}_{G+\sum ^{n}_{i=1}e_{i}}|_O(z)$  is not  even-interpolating.

\end{enumerate}

\end{thm}
\begin{proof}Suppose that $~^{\partial}{\Eu}^{\times}_{G+\sum ^{n}_{i=1}e_{i}}|_O(z)$  is even-interpolating. We
observe that the number of faces of $\eu^{0}_{min}(G +\sum ^{n}_{i=1}e_{i})$ is    $f^{0}_{max}(G +\sum ^{n}_{i=1}e_{i})$, and the number of faces of $\eu^{1}_{max}(G +\sum ^{n}_{i=1}e_{i})$ is    $f^{1}_{min}(G +\sum ^{n}_{i=1}e_{i})$. Let $e\in A_0$, $e\notin A_1$,
 assume that both $G^{\times|_{A_{0}}}$ and  $G^{\times|_{A_{1}}}$ are  orientable ribbon graphs, then $f(G^{\times|_{A_{0}}})=f^{0}(G^{\times|_{A_{0}}})$, $f(G^{\times|_{A_{1}}})=f^{1}(G^{\times|_{A_{1}}})$.

For Item (1), we have  \begin{eqnarray}
f^{1}(G^{\times|_{A_{1}}})= f^{0}(G^{\times|_{A_{0}}})= f^{1}_{min}(G)= f^{0}_{max}(G).\label{tm1} \end{eqnarray}
 Observe that
$
f^{0}_{max}(G +e_{1})=
 f^{0}_{max}(G)+1,$ or
$f^{0}_{max}(G +e_{1})=
 f^{0}_{max}(G)-1,$
which implies that we should consider the following two cases.

$\mathbf{Case~ 1}:$
 $f^{0}_{max}(G +e_{1})=f^{0}_{max}(G)+1$. Note that  deleting an even number of twisted multiple edges does not change the number of faces. Then,  when $n$ is even, the maximum value of $f^{0}(G +\sum ^{n}_{i=1}e_{i})$ is the maximum value of  $f^{0}(G )$, that is
 \begin{eqnarray}
f^{0}_{max}(G +\sum ^{n}_{i=1}e_{i})&=& f^{0}_{max}(G ). \label{tm88}\end{eqnarray}

When $n$ is odd, and deleting  $n-1$  twisted multiple edges does not change the number of faces.
 We have
\begin{eqnarray}
f^{0}_{max}(G +\sum ^{n}_{i=1}e_{i})&=& f^{0}_{max}(G +e_{1})\label{tm8}\\
&=& f^{0}_{max}(G)+1.\notag \end{eqnarray}

By  Euler's formula, it follows that
 \begin{eqnarray}
\eu^{0}_{min}(G +\sum ^{n}_{i=1}e_{i})&=& 2c(G)+e(G )+n-v(G )- f^{0}_{max}(G)\quad\text{by~(\ref{tm88})}\notag\\
&=& 2c(G)+e(G )+n-v(G )- f^{0}(G^{\times|_{A_{0}}})\quad\text{by~(\ref{tm1})} \notag\\
&=& \eu^{0}(G^{\times|_{A_{0}}})+n,\label{tm08}
 \end{eqnarray} for even $n$, and
 \begin{eqnarray}
\eu^{0}_{min}(G +\sum ^{n}_{i=1}e_{i})&=& 2c(G)+e(G )+n-v(G )- f^{0}_{max}(G)-1\quad\text{by~(\ref{tm8})}\notag\\
&=& 2c(G)+e(G )+n-v(G )- f^{0}(G^{\times|_{A_{0}}})-1\quad\text{by~(\ref{tm1})} \notag\\
&=& \eu^{0}(G^{\times|_{A_{0}}})+n-1,\label{tm03}
 \end{eqnarray} for odd $n$.

Let    $A'_{1}=A_{1}$. It is easy to see that $(G +\sum ^{n}_{i=1}e_{i})^{\times|_{A'_{1}}}$ is an orientable ribbon graph,
 and  $f((G +\sum ^{n}_{i=1}e_{i})^{\times|_{A'_{1}}})=f^{1}((G +\sum ^{n}_{i=1}e_{i})^{\times|_{A'_{1}}})$.
      Clearly, for any ribbon subset $A'_{1}$ of $G +\sum ^{n}_{i=1}e_{i}$, we have
$
f^{1}((G +\sum ^{n}_{i=1}e_{i})^{\times|_{A'_{1}}})=  f^{1}(G^{\times|_{A_{1}}})+n.
$
Furthermore, \begin{eqnarray}
f^{1}_{min}(G +\sum ^{n}_{i=1}e_{i})&=&f^{1}_{min}(G)+n,\label{tm111}
\end{eqnarray}

similarly, we have
\begin{eqnarray}
\eu^{1}_{max}(G +\sum ^{n}_{i=1}e_{i})&=& 2c(G)+e(G )+n-v(G )- f^{1}_{min}(G)-n\quad\text{by (\ref{tm111})}\notag\\
&=& 2c(G)+e(G )+n-v(G )- f^{0}(G^{\times|_{A_{0}}})-n \quad\text{by (\ref{tm1})}\notag\\
&=& \eu^{0}(G^{\times|_{A_{0}}}).\label{tm3}
 \end{eqnarray}
 Recall
that $~^{\partial}{\Eu}^{\times}_{G+\sum ^{n}_{i=1}e_{i}}|_O(z)$  is even-interpolating, so,
 there is a subinterval
   $[\eu^{1}_{max}(G +\sum ^{n}_{i=1}e_{i}),
   \eu^{0}_{min}(G +\sum ^{n}_{i=1}e_{i})]$   in $supp(~^{\partial}{\Eu}^{\times}_{G+\sum ^{n}_{i=1}e_{i}}|_O(z))$. Note that the Euler-genus of the orientable ribbon graph $(G +\sum ^{n}_{i=1}e_{i})^{\times|_A}$ is either smaller than $\eu^{1}_{max}(G +\sum ^{n}_{i=1}e_{i})$ or larger than $\eu^{0}_{min}(G +\sum ^{n}_{i=1}e_{i})$, so, there are only two numbers $\eu^{1}_{max}(G +\sum ^{n}_{i=1}e_{i})$,
    $ \eu^{0}_{min}(G +\sum ^{n}_{i=1}e_{i})$ in $[\eu^{1}_{max}(G +\sum ^{n}_{i=1}e_{i}),
    \eu^{0}_{min}(G +\sum ^{n}_{i=1}e_{i})]$.

  { If $n$ is odd,  from (\ref{tm03}) and (\ref{tm3}), the size of  $[\eu^{1}_{max}(G +\sum ^{n}_{i=1}e_{i}),
    \eu^{0}_{min}(G +\sum ^{n}_{i=1}e_{i})]$ is
      $\frac{\eu^{0}(G^{\times|_{A_{0}}})+n-1-\eu^{0}(G^{\times|_{A_{0}}})}{2}+1$.
 Otherwise,   from (\ref{tm08}) and (\ref{tm3}), the size of  $[\eu^{1}_{max}(G +\sum ^{n}_{i=1}e_{i}),
    \eu^{0}_{min}(G +\sum ^{n}_{i=1}e_{i})]$ is
 $\frac{\eu^{0}(G^{\times|_{A_{0}}})+n-\eu^{0}(G^{\times|_{A_{0}}})}{2}+1$. When the size of $[\eu^{1}_{max}(G +\sum ^{n}_{i=1}e_{i}),
    \eu^{0}_{min}(G +\sum ^{n}_{i=1}e_{i})]$ is greater than or equal to 3,   it is  contrary to the hypothesis that $~^{\partial}{\Eu}^{\times}_{G+\sum ^{n}_{i=1}e_{i}}|_O(z)$  is even-interpolating. Therefore,
    $~^{\partial}{\Eu}^{\times}_{G+\sum ^{n}_{i=1}e_{i}}|_O(z)$ is not  even-interpolating for
 $n\geq4$.}

$\mathbf{Case~ 2}:$  $f^{0}_{max}(G +e_{1})=f^{0}_{max}(G)-1$.  Similarly, when $n$ is odd, \begin{eqnarray}
f^{0}_{max}(G +\sum ^{n}_{i=1}e_{i})&=& f^{0}_{max}(G +e_{1})\label{tm9}\\
&=& f^{0}_{max}(G)-1.\notag \end{eqnarray}
\begin{eqnarray}
\eu^{0}_{min}(G +\sum ^{n}_{i=1}e_{i})&=& 2c(G)+e(G )+n-v(G )- f^{0}_{max}(G)+1\quad\text{by~(\ref{tm9})}\notag\\
&=& 2c(G)+e(G )+n-v(G )- f^{0}(G^{\times|_{A_{0}}})+1\quad\text{by~(\ref{tm1})} \notag\\
&=& \eu^{0}(G^{\times|_{A_{0}}})+n+1,\label{tm10}
 \end{eqnarray}
  {If $n$ is odd,  from (\ref{tm10}) and (\ref{tm3}),
 $\frac{\eu^{0}(G^{\times|_{A_{0}}})+n+1-\eu^{0}(G^{\times|_{A_{0}}})}{2}+1$ is the size of  $[\eu^{1}_{max}(G +\sum ^{n}_{i=1}e_{i}),
    \eu^{0}_{min}(G +\sum ^{n}_{i=1}e_{i})]$.
 The remainder of the argument is analogous to that in case 1 and omitted. Hence, $~^{\partial}{\Eu}^{\times}_{G+\sum ^{n}_{i=1}e_{i}}|_O(z)$ is not  even-interpolating  for $n\geq3$.}

For item (2),  we shall adopt the same procedure as in the proof of  case~1.

\noindent When $n$ is odd,
  by (\ref{tm8}),    we have $f^{0}_{max}(G +\sum ^{n}_{i=1}e_{i})= f^{0}_{max}(G+e_{1} )$, and \begin{eqnarray}\label{1121}
\eu^{0}_{min}(G +\sum ^{n}_{i=1}e_{i})&=& 2c(G+e_{1} )+e(G+e_{1} )+n-1-v(G+e_{1} )- f^{0}_{max}(G+e_{1} )\notag\\
&=& 2c(G+e_{1} )+e(G+e_{1} )-v(G+e_{1} )- f^{0}_{max}(G+e_{1} )+n-1\notag\\
&=& \eu^{0}_{min}(G +e_{1})+n-1.
 \end{eqnarray}
\noindent When $n$ is even, by (\ref{tm88}),    we have $f^{0}_{max}(G +\sum ^{n}_{i=1}e_{i})
= f_{max}^{0}(G)$, and
\begin{eqnarray}\label{1122}
\eu^{0}_{min}(G +\sum ^{n}_{i=1}e_{i})&=& 2c(G)+e(G)+n-v(G)- f^{0}_{max}(G)\notag\\
&=& 2c(G)+e(G)-v(G)- f^{0}_{max}(G)+n\notag\\
&=& \eu^{0}_{min}(G)+n.
 \end{eqnarray}
By (\ref{tm3}), we have
\begin{eqnarray}\label{1123}
\eu^{1}_{max}(G +\sum ^{n}_{i=1}e_{i})&=& 2c(G)+e(G)+n-v(G)- f^{0}_{min}(G)-n\notag\\
&=& 2c(G)+e(G)-v(G)- f^{0}_{min}(G)\notag\\
&=& \eu^{1}_{max}(G).
 \end{eqnarray}

Recall that there exists a subinterval $[\eu^{1}_{max}(G +\sum ^{n}_{i=1}e_{i}),
   \eu^{0}_{min}(G +\sum ^{n}_{i=1}e_{i})]$ in $supp(~^{\partial}{\Eu}^{\times}_{G+\sum ^{n}_{i=1}e_{i}}|_O(z))$. By (\ref{1121})-(\ref{1123}),
 there exist two subintervals  $[\eu^{1}_{max}(G),$ $\eu^{0}_{min}(G +e_{1})+n-1]$ and
   $[\eu^{1}_{max}(G),\eu^{0}_{min}(G)+n]$   in $supp(~^{\partial}{\Eu}^{\times}_{G+\sum ^{n}_{i=1}e_{i}}|_O(z))$. When

  \begin{equation}\label{1129}
 \frac{\eu^{0}_{min}(G )+n-\eu^{1}_{max}(G)}{2}+1\geq 3,
 \end{equation}
 or
   \begin{equation}\label{1124}
  \frac{\eu^{0}_{min}(G +e_{1})+n-1-\eu^{1}_{max}(G)}{2}+1\geq 3,
\end{equation}
{there are at least three numbers in $[\eu^{1}_{max}(G),\eu^{0}_{min}(G +e_{1})+n-1]$ or $[\eu^{1}_{max}(G),\\
\eu^{0}_{min}(G)+n]$ respectively.  However, there are only two numbers $\eu^{1}_{max}(G)$, $\eu^{0}_{min}(G +e_{1})+n-1$ in  $[\eu^{1}_{max}(G),
\eu^{0}_{min}(G +e_{1})+n-1]$ and two numbers $\eu^{1}_{max}(G)$, $\eu^{0}_{min}(G)+n$  in $[\eu^{1}_{max}(G),
\eu^{0}_{min}(G)+n]$,  it is impossible.

Combining (\ref{1121})-(\ref{1124}), we  finish the proof.}

\end{proof}

\begin{thm}\cite{GMT21a}\label{join}
Let $G=G_{1}\vee G_{2}$. Then
\begin{eqnarray*}
~^{\partial}{\Eu}^{\times}_{G}(z) &=& ~^{\partial}{\Eu}^{\times}_{G_{1}}(z) ~^{\partial}{\Eu}^{\times}_{G_{2}}(z).\label{sub1}
\end{eqnarray*}

\end{thm}

\begin{rem}
{It is easy to see that Theorem \ref{sub} still holds for $~^{\partial}{\Eu}^{\times}_{G}|_O(z)$.
}
\end {rem}

Now, let's take two examples to illustrate the above results.

\begin{figure}[h]
  \centering
  \includegraphics[width=0.2\textwidth]{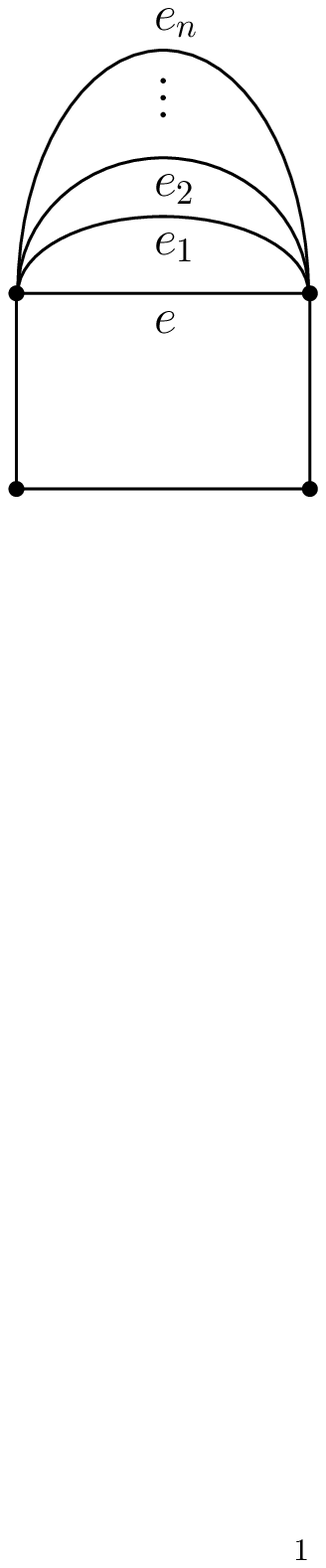}
 \caption{$C_4+\sum ^{n}_{i=1}e_{i}$ }
 \label{d4ei}
\end{figure}

\begin{example}{
Let $C_m$ be the planar ribbon $m$-cycle. Evidently,
$~^{\partial}{\Eu}^{\times}_{C_2}|_O(z)=2.$
Using the fact that  $C_m$ can be  obtained
from $C_2$ by   subdividing a ribbon   $m-2$ times, we have
\begin{eqnarray*}
~^{\partial}{\Eu}^{\times}_{C_m}|_O(z)  &=&2^{m-2}  ~^{\partial}{\Eu}^{\times}_{C_2}|_O(z)\quad \text{by Theorem \ref{sub}}\notag\\
&=&2^{m-1}.
\end{eqnarray*}
The ribbon graph $C_m+\sum ^{n}_{i=1}e_{i}$ is obtained from $C_m$ by attaching $n$ parallel  edges $e_1$, $e_2$, $\cdots$, $e_n$  to the  ribbon $e$, as shown in Figure \ref{d4ei}. It is clear that $f^{0}(C_m^{\times|_{A}})=2$. Since the number of faces of orientable ribbon graph $(C_m+e_1)^{\times|_{B}}$ with $e\in B$ is 1, it follows that
$f^{0}_{max}(C_m+e_1)=f^{0}(C_m^{\times|_{A}})-1=1$.  By  Theorem \ref{tm} (1), when $n\geq 3,$  the restricted-orientable partial-$\times$ polynomial of $C_m+\sum ^{n}_{i=1}e_{i}$ is not even-interpolating.}

Let $H=D_{n}\vee C_2$,  then by Theorem
  \ref{join}  and  Example \ref{dnn},
the restricted-orientable partial-$\times$ polynomial of  $H$ is not even-interpolating  for $n\geq 5$.

\end{example}

\begin{figure}[h]
  \centering
  \includegraphics[width=0.5\textwidth]{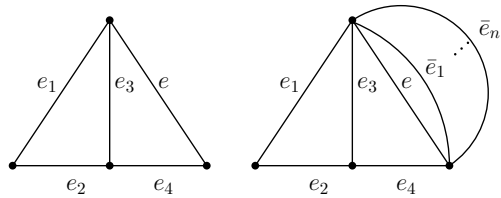}
 \caption{$G$ and $G+\sum ^{n}_{i=1}\bar{e}_{i}$ }
 \label{f111}
\end{figure}
\begin{example}
{Let $G$ and $G+\sum ^{n}_{i=1}\bar{e}_{i}$ be the ribbon graphs  in Figure  \ref{f111}. When $A\in \{\{e_{1},e_{3},e\},\{e_{2},e_{3},e\},\{e_{1},e_{3},e_{4}\},\{e_{2},e_{3},e_{4}\}\}$,
the number of faces of orientable ribbon graph $G^{\times|_{A}}$ is 1.
When $A\in \{\{e_{1},e_{2}\},\{e_{4},e\},\{e_{1},e_{2},e_{4},e\},\emptyset\}$,
the number of faces of orientable ribbon graph $G^{\times|_{A}}$ is 3.
We infer that
$~^{\partial}{\Eu}^{\times}_{G}|_O(z)=4+4z^{2},$  $f^{0}_{max}(G)=3$,} and
$f^{1}_{min}(G)=1$.  When $B\in \{\{e_{1},e_{3},e,\bar{e}_{1} \},\{e_{2},e_{3},e, \bar{e}_{1}\},$ $\{e_{1},e_{2},e_{4},e,  \bar{e}_{1}\},\{e_{4},e,  \bar{e}_{1}\}\}$,
we have $f((G+\bar{e}_{1})^{\times|_{B}})=2$, thus,
\begin{eqnarray*}
 f^{0}_{max}(G+\bar{e}_{1})  &=&2 ~<~ f^{0}_{max}(G).
\end{eqnarray*}
 Therefore, by  Theorem \ref{tm} (2),
 when $n \geq2-1+4= 5,$    the restricted-orientable partial-$\times$ polynomial of $G+\sum ^{n}_{i=1}\bar{e}_{i}$ is not even-interpolating.

\end{example}

\subsection{Acknowledgments}

We are grateful to the anonymous referees for their valuable comments.




\begin{thebibliography}{99}


\bibitem[AE19]{AE19} L. Abrams, and J. Ellis-Monaghan, New dualities from old: generating geometric, Petrie, and Wilson
dualities and trialities of ribbon graphs, arXiv:1901.03739v2 [math.CO], 9 Aug (2019).

\bibitem[AVJ99]{AVJ99}
A. Brandstadt, V. B. Le, and J. P. Spinrad, Graph Classes: a Survey, \emph{Society for
Industrial and Applied Mathematics (SIAM)}, Philadelphia, PA, (1999).


 \bibitem[Bol02]{Bol02}
B. Bollob$\acute{a}$s, O. Riordan, A polynomial of graphs on surfaces, \emph{Math. Ann.} 323 (2002) 81-96.

 \bibitem[Bry71]{Bry71} T. Brylawski, A combinatorial model for series-parallel networks,
 \emph{Trans. Amer. Math. Soc.}, \textbf{154}  (1971), 1-22.



\bibitem[Chm09]{Chm09}
 S. Chmutov, Generalized duality for graphs on surfaces and the signed Bollob\'as-Riordan polynomial, \emph{J. Combin. Theory Ser. B } \textbf{99} (2009), 617-638.

\bibitem[EM12]{EM12}
J. Ellis-Monaghan and I. Moffatt,
Twisted duality for embedded graphs,
\textsl{Trans. Amer. Math. Soc.} \textbf{364} (2012), 1529--1569.

\bibitem[EM13]{EM13}
 J. Ellis-Monaghan and I. Moffatt, Graphs on Surfaces: Dualities, Polynomials, and Knots. \emph{Springer,} (2013).

 \bibitem[GT87]{GT87}
 J. L. Gross and T. W. Tucker, Topological Graph Theory, John Wiley $\&$
Sons, Inc. New York, (1987).

\bibitem[GMT20]{GMT20}
 J. L. Gross, T. Mansour and T. W. Tucker,
Partial duality for ribbon graphs, I:Distributions,
\emph{European J. Combin.} \textbf{86} (2020),  103084.

 \bibitem[GMT21a]{GMT21a}
 J. L. Gross, T. Mansour and T. W. Tucker,
Partial duality for ribbon graphs, II: Partial-twuality polynomials and monodromy computations,
\textsl{European J. Combin.} \textbf{95} (2021), 103329.


\bibitem[GMT21b]{GMT21b}
 J. L. Gross, T. Mansour and T. W. Tucker,
Partial duality for ribbon graphs, III: a Gray code algorithm for enumeration,
\textsl{J Algebr. Comb.} (2021). https://doi.org/10.1007/s10801-021-01040-y.

\bibitem[Mof12]{Mof12}
I. Moffatt, A characterization of partially dual graphs, \emph{Journal of Graph Theory } 67(3) (2010), 198-217.

\bibitem[Tut54]{Tut54}
W. T. Tutte,  A contribution to the theory of chromatic polynomials,
\emph{ Can. J. Math.}, \textbf{6} (1954), 80-91


\end{thebibliography}
\end{document}